\documentclass[runningheads]{llncs}
\usepackage{graphicx}

\usepackage{ulem,amsmath,amssymb,amsfonts,amsthm}
\usepackage[utf8]{inputenc}
\usepackage{tikz}
\usepackage{import}
\usepackage{subcaption}
\usepackage[above,below,section]{placeins}

\newcommand{\reflemma}[1]{Lemma~\ref{#1}}

\PassOptionsToPackage{hyphens}{url}\usepackage{hyperref}

\usepackage{mathtools}
\DeclarePairedDelimiter\norm\lVert\rVert

\DeclareMathOperator*{\argmin}{arg\,min}

\makeatletter
\newcommand{\fancybreak}[1]{%
  \penalty -100
  \noindent\parbox{\linewidth}{\centering #1}%%\null
  \par
  \@afterindentfalse
  \@afterheading}
\makeatother

\begin{document}
\title{Simulation of Conditioned Diffusions on the Flat Torus\thanks{MHJ, AM, and SS are supported by the CSGB Centre for Stochastic Geometry and Advanced Bioimaging funded by a grant from the Villum Foundation.}}
%
%\titlerunning{Abbreviated paper title}
% If the paper title is too long for the running head, you can set
% an abbreviated paper title here
%
\author{Mathias Højgaard Jensen \and
Anton Mallasto \and
Stefan Sommer}
\authorrunning{Højgaard Jensen et al.}
% First names are abbreviated in the running head.
% If there are more than two authors, 'et al.' is used.
%
\institute{Department of Computer Science, Copenhagen University, Denmark \\
\email{\{matje,mallasto,sommer\}@di.ku.dk}}
\maketitle              % typeset the header of the contribution
\begin{abstract}
Diffusion processes are fundamental in modelling stochastic dynamics in natural sciences. Recently, simulating such processes on complicated geometries has found applications for example in biology, where toroidal data arises naturally when studying the backbone of protein sequences, creating a demand for efficient sampling methods.
In this paper, we propose a method for simulating diffusions on the flat torus, conditioned on hitting a terminal point after a fixed time, by considering a diffusion process in $\mathbb{R}^2$ which we project onto the torus. We contribute a convergence result for this diffusion process, translating into convergence of the projected process to the terminal point on the torus. We also show that under a suitable change of measure, the Euclidean diffusion is locally a Brownian motion.

\end{abstract}
\section{Introduction}

Stochastic differential equations are ubiquitous in models describing evolution of dynamical systems with, e.g. in modelling the evolution of DNA or protein structure, in pricing financial derivatives, or for modelling changes in landmark configurations which are essential in shape analysis and computational anatomy. In settings where the beginning and end values are known on some fixed time interval, the use of Brownian bridges becomes natural to evaluate the uncertainty on the intermediate time interval. 

When the data elements are elements of non-linear spaces, here differentiable manifolds, methodology for simulating bridge processes is lacking. In particular, in cases where the transition probability densities are intractable, it is of interest to use simulation schemes that can numerically approximate the true densities. In this paper we propose a method for simulating diffusion bridges on the flat torus, $\mathbb{T}^2=\mathbb{R}^2/\mathbb{Z}^2$, i.e. we propose a process that can easily be simulated and satisfies that the distribution of the true bridge of interest is absolutely continuous with respect to the distribution of this proposal process. This specific case will serve as an example of the more general setting of simulating diffusion bridge processes on Riemannian manifolds. Because of the non-trivial topology of the torus $\mathbb{T}^2$, the conditioned process will be equivalent to a process in $\mathbb{R}^2$ that is conditioned on ending up in a set of points. Therefore, we will address the question of conditioning a process on infinitely many points. Secondly, we will handle the case when the process crosses the cut locus of the target point, i.e. the set of points with no unique distance minimizing geodesic.

It is a basic consequence of Doob's h-transform that the distribution of a conditioned diffusion process is the same as another diffusion process with the drift depending on the transition density. However, as mentioned in \cite{delyon_simulation_2006}, using this transform directly is undesirable for simulation purposes as the transition density is often intractable. Instead, the authors introduce a diffusion process which can easily be simulated and with the property that the distribution of the true conditioned diffusion is absolutely continuous wrt. the diffusion used for simulation. We here use this approach that in \cite{delyon_simulation_2006} covers the Euclidean case as the starting point for developing a simulation scheme on the torus.

\fancybreak{}

Recent papers have considered diffusion processes on the torus, for example, Langevin diffusions on the torus were studied in \cite{garcia-portugues_langevin_2017} and \cite{golden_generative_2017}, in the latter to describe protein evolution. In this paper, we introduce a diffusion process in $\mathbb{R}^2$ which can easily be simulated and projects onto a bridge process on the torus. More generally, Brownian bridges on manifolds have been studied for example in the context of landmark manifolds \cite{sommer2017bridge} and used for approximating the transition density of the Brownian motion. The present paper uses bridges on the flat torus to exemplify how some of the challenges of bridge simulation on Riemannian manifolds can be addressed, here in particular non-trivial topology of the manifold.

\fancybreak{}

We begin in Section \ref{sec: theoretical setup} with a short introduction to Brownian bridge processes in the standard Eucliden case and how it relates to the definition of a Brownian bridge process on the flat torus. At the end we introduce the stochastic differential equation (SDE) which will be used for simulating the bridge process. In Section \ref{sec: existence of solution} we argue that a strong solution of our proposed SDE exist. We show results about convergence and absolute continuity in Section \ref{sec: absolute continuity}. Numerical examples are presented in Section \ref{sec: numerical experiments}.

\section{Theoretical Setup}\label{sec: theoretical setup}

This section will briefly review some Brownian bridges theory and discuss the torus case. A more general theory of diffusion bridges can be found in \cite{delyon_simulation_2006}, constituting the main reference for this work. At the end, we introduce our proposal process.

\fancybreak{}

Consider a Brownian motion $W = (W_t)_{t \geq 0}$ in $\mathbb{R}^n$. By conditioning, it can be shown that $W$ will end up at a given point at a given time. For example, the process given by $B_t = W_t - \frac{t}{T} W_T$ defines a Brownian bridge conditioned to return to $0$ at time $T$. It can be shown that the diffusion process given by
\begin{equation}\label{standard brownian bridge}
  dX_t = \frac{b-X_t}{T-t}dt + dW_t; \qquad 0 \leq t < T \quad
  \text{and} \quad X_0 = a,
\end{equation}
for given $a,b \in \mathbb{R}^d$ and $W$ a $d$-dimensional standard Brownian motion, is a $d$-dimensional Brownian bridge from $a$ to $b$ on $[0,T]$ (see e.g. \cite[sec. 5.6]{karatzas1991brownian}). More generally, diffusion bridges can be defined through Doob's $h$-transform, that is, the distribution of a diffusion
\begin{align*}
  dX_t = b(t,X_t)dt + \sigma(t,X_t)dW_t, \qquad X_0=a,
\end{align*}
conditioned on $X_T=b$ is the same as that of 
\begin{align*}
 dY_t &= \tilde{b}(t,Y_t)dt + \sigma(t,Y_t)dW_t,\\
 \tilde{b}(t,x) &= b(t,x) + \sigma(t,x)\sigma^T(t,x)\nabla_x     	\log(p(t,x;T,b)),
\end{align*}
where $p(t,x;T,b)$ denotes the transition density of the process $X$. In the usual setting where $p$ is the transition density of a Brownian motion it has the form
\begin{align*}
 p(s,x;t,y) = \frac{1}{\sqrt{2\pi (t-s)}}\exp\biggl(-\frac{||
 x-y||^2}{2(t-s)}\biggr), \qquad s<t,
\end{align*}
which yields \eqref{standard brownian bridge}.

We propose a method similar to the Euclidean scheme \cite{delyon_simulation_2006} for simulating Brownian bridges on the flat torus, which is of the form
\begin{equation}\label{diffusion on torus}
 dX_t = b(t,X_t)dt + \sigma dW_t; \quad 0 \leq t < T \quad  \text{ and } \quad X_0 = a \quad \text{ a.s.},
\end{equation}
where $\sigma > 0$, $a \in \mathbb{T}^2$ is given, and $W$ is a two-dimensional standard Brownian motion. The exact form of $b(t,x)$ will become apparent below. It is important here to note that in the particular case of the flat torus the transition density for the Brownian motion is known and therefore it is possible to simulate from the distribution of the true Brownian bridge on $\mathbb{T}^2$, however, it requires the calculation of the distance to infinitely many points which the proposed model does not. In Figure \ref{fig: true vs proposed} is shown paths of the proposed model and the corresponding paths of the true bridge process. 

\fancybreak{}

Let $\pi \colon \mathbb{R}^2 \rightarrow \mathbb{T}^2=\mathbb{R}^2/\mathbb{Z}^2$ denote the canonical projection onto the torus. The standard two-dimensional Brownian motion $W=(W^1, W^2)$, for two independent one-dimensional Brownian motions $W^1$ and $W^2$, is mapped to a Brownian motion $B=(B_t)_{t \geq 0}$ on the flat torus $\mathbb{T}^2$ by the projection map $\pi$. Indeed, we can identify the torus $\mathbb{T}^2$ with the unit cube $Q = \{x \in \mathbb{R}^2 : - \frac{1}{2} \leq x_k < \frac{1}{2}, k=1,2\}$. Then for $g \in C^{\infty}(\mathbb{T}^2)$ the Laplace-Beltrami operator, $\Delta_{\mathbb{T}^2}$, on $\mathbb{T}^2$ corresponds to the restriction to $Q$ of the usual Euclidean Laplacian, $\Delta_{\mathbb{R}^2} \tilde{g}$, where $\tilde{g}$ denotes the periodic extension of $g$, i.e. $\tilde{g} = g \circ \pi$ (see \cite[Sec. 3.5]{sogge_hangzhou_2014}). Since $W$ is a Brownian motion in $\mathbb{R}^2$ if and only if it satisfies the diffusion equation
\begin{align*}
 h(W_t) \overset{m}{=} h(W_0) - \frac{1}{2} \int_0^t 
 \Delta_{\mathbb{R}^2} h(W_s) ds,
\end{align*}
for all smooth functions $h$, where $X\overset{m}{=}Y$ means that the difference $X-Y$ is a local martingale (see e.g. \cite[Sec. 1.5]{emery1989stochastic}), it follows that, for $h = \tilde{g}$,
\begin{align*}
 \tilde{g}(W_t) & \overset{m}{=} \tilde{g}(W_0) - \frac{1}{2} 
 \int_0^t \Delta_{\mathbb{R}^2} \tilde{g}(W_s)ds = g (B_0) -  
 \frac{1}{2}\int_0^t \Delta_{\mathbb{T}^2} g(B_s)ds  
 \overset{m}{=} g(B_t).
\end{align*}
As this holds for all smooth functions $g$ on $\mathbb{T}^2$, we get that $B$ is a Brownian motion on $\mathbb{T}^2$ in agreement with the definition of a manifold-valued Brownian motion given in \cite[Sec. 3.2]{hsu2002stochastic}. 

By conditioning $B$ on $\mathbb{T}^2$ to hit a given point $a \in \mathbb{T}^2$, at some fixed time $0 \leq T < \infty$, it is seen that
\begin{align*}
 \{\omega \in \Omega : B_T(\omega) = a\} = \{\omega \in 
 \Omega : W_T(\omega) \in \pi^{-1}(a) \},
\end{align*}
and so simulating a Brownian bridge on the flat torus $\mathbb{T}^2$ is equivalent to simulating a two-dimensional standard Brownian motion conditioned to end up in the set $\pi^{-1}(a)$ at time $T$. The diffusion given by \eqref{standard brownian bridge} will not suffice as it is constructed to hit exactly one point. It will, however, provide one subset of sample paths of the Brownian bridge on $\mathbb{T}^2$, corresponding to subset of paths that will "unwrap" the same number of times that it "wraps" around the cut locus. This is illustrated in Figure \ref{fig:example}. To give a precise meaning to this statement we consider the $h$-transform
\begin{equation*}
 h(t,z) = \sum_{y \in \pi^{-1}(a)} \frac{p(t,z;T,y)}
 {p(0,z_0;T,y)},
\end{equation*}
with $p$ denoting the transition density of the two-dimensional Brownian motion, which by Doob's $h$-transform implies that the distribution of $W$ conditioned on $W_T \in \pi^{-1}(a)$ is the same as the distribution of the diffusion 
\begin{equation}\label{true bridge process}
\begin{split}
 dZ_t &= \sigma^2 \nabla_z \log\bigg(\sum_{y \in \pi^{-1}(a)} 
 p(t,z;T,y)\bigg)\bigg|_{x=Z_t}dt + \sigma dW_t\\
 &=  \sum_{y \in \pi^{-1}(a)} g_y(t,Z_t)\frac{y-Z_t}{T-t} dt 
 + \sigma dW_t, \qquad Z_0 = z_0,
\end{split}
\end{equation}
where 
\begin{equation*}
  g_y(t,x) =
  \frac{\exp\big(-\frac{\norm{y-z}^2}{2\sigma^2(T-t)}\big)}{\sum_{y
      \in
      \pi^{-1}(a)}\exp\big(-\frac{\norm{y-z}^2}{2\sigma^2(T-t)}\big)}.
\end{equation*}

Instead, we propose to consider the diffusion process on $[0,T)$, for some fixed positive $T$, defined by 
\begin{equation}\label{bridge diffusion on torus}
 dX_t = 1_{G^c}(X_t) \frac{\alpha(X_t) - X_t}{T-t}dt + \sigma  
 dW_t, \qquad X_0 = x_0
\end{equation}
where $\sigma > 0$ and $\alpha$ is defined by
\begin{align*}
\alpha(X_t) &= \argmin_{y \in \pi^{-1}(a)} \norm{y-X_t},
\end{align*} 
with $a\in \mathbb{T}^2$, and where $G$ is the set of "straigt lines" of the form $\mathbb{R} \times \{x\}$ (resp. $\{x\}\times \mathbb{R}$) in $\mathbb{R}^2$ where $\alpha(X_t)$ is not unique (see Figure \ref{fig:example}). The indicator function removes the drift when the process does not have a natural attraction point.

\begin{figure}[!bt]
\centering
  \def\svgwidth{.7\columnwidth}
  %% Creator: Inkscape inkscape 0.92.3, www.inkscape.org
%% PDF/EPS/PS + LaTeX output extension by Johan Engelen, 2010
%% Accompanies image file 'R2-torus-vs2.eps' (pdf, eps, ps)
%%
%% To include the image in your LaTeX document, write
%%   \input{<filename>.pdf_tex}
%%  instead of
%%   \includegraphics{<filename>.pdf}
%% To scale the image, write
%%   \def\svgwidth{<desired width>}
%%   \input{<filename>.pdf_tex}
%%  instead of
%%   \includegraphics[width=<desired width>]{<filename>.pdf}
%%
%% Images with a different path to the parent latex file can
%% be accessed with the `import' package (which may need to be
%% installed) using
%%   \usepackage{import}
%% in the preamble, and then including the image with
%%   \import{<path to file>}{<filename>.pdf_tex}
%% Alternatively, one can specify
%%   \graphicspath{{<path to file>/}}
%% 
%% For more information, please see info/svg-inkscape on CTAN:
%%   http://tug.ctan.org/tex-archive/info/svg-inkscape
%%
\begingroup%
  \makeatletter%
  \providecommand\color[2][]{%
    \errmessage{(Inkscape) Color is used for the text in Inkscape, but the package 'color.sty' is not loaded}%
    \renewcommand\color[2][]{}%
  }%
  \providecommand\transparent[1]{%
    \errmessage{(Inkscape) Transparency is used (non-zero) for the text in Inkscape, but the package 'transparent.sty' is not loaded}%
    \renewcommand\transparent[1]{}%
  }%
  \providecommand\rotatebox[2]{#2}%
  \newcommand*\fsize{\dimexpr\f@size pt\relax}%
  \newcommand*\lineheight[1]{\fontsize{\fsize}{#1\fsize}\selectfont}%
  \ifx\svgwidth\undefined%
    \setlength{\unitlength}{348.96695835bp}%
    \ifx\svgscale\undefined%
      \relax%
    \else%
      \setlength{\unitlength}{\unitlength * \real{\svgscale}}%
    \fi%
  \else%
    \setlength{\unitlength}{\svgwidth}%
  \fi%
  \global\let\svgwidth\undefined%
  \global\let\svgscale\undefined%
  \makeatother%
  \begin{picture}(1,0.4396127)%
    \lineheight{1}%
    \setlength\tabcolsep{0pt}%
    \put(0,0){\includegraphics[width=\unitlength]{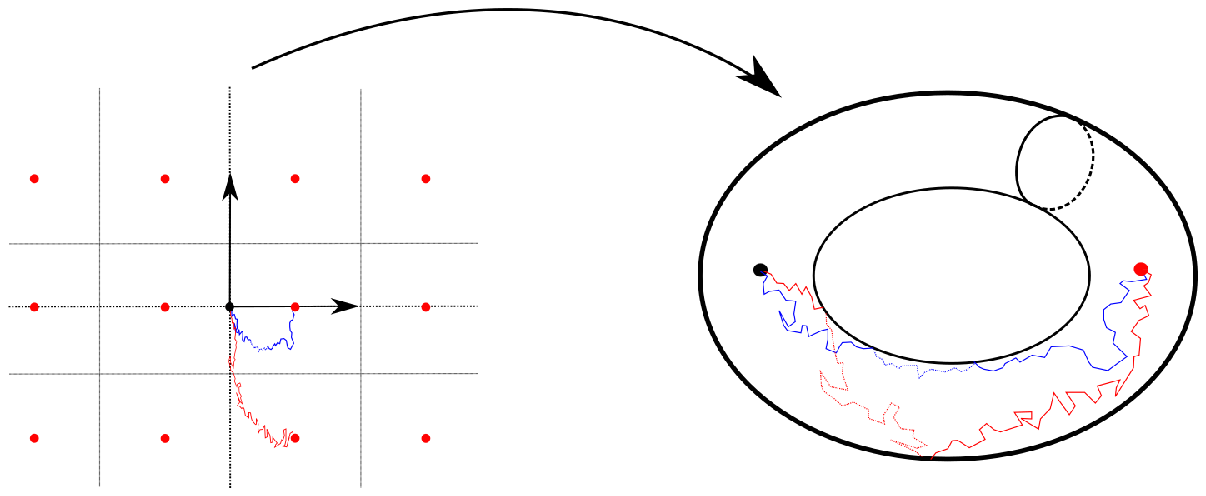}}%
    \put(0.38587584,0.38046088){\color[rgb]{0,0,0}\makebox(0,0)[lt]{\lineheight{1.25}\smash{\begin{tabular}[t]{l}$\pi$\end{tabular}}}}%
    \put(0.10241398,0.37579491){\color[rgb]{0,0,0}\makebox(0,0)[lt]{\lineheight{1.25}\smash{\begin{tabular}[t]{l}$\mathbb{R}^2$\end{tabular}}}}%
    \put(0.23743041,0.09886392){\color[rgb]{0,0,0}\makebox(0,0)[lt]{\lineheight{1.25}\smash{\begin{tabular}[t]{l}$X_t(\omega)$\end{tabular}}}}%
    \put(0.70200127,-0.00026806){\color[rgb]{0,0,0}\makebox(0,0)[lt]{\lineheight{1.25}\smash{\begin{tabular}[t]{l}$B_t(\omega)$\end{tabular}}}}%
    \put(0.73927052,0.37924769){\color[rgb]{0,0,0}\makebox(0,0)[lt]{\lineheight{1.25}\smash{\begin{tabular}[t]{l}$\mathbb{T}^2$\end{tabular}}}}%
    \put(0.06404508,0.22470069){\color[rgb]{0,0,0}\makebox(0,0)[lt]{\lineheight{1.25}\smash{\begin{tabular}[t]{l}$G$\end{tabular}}}}%
  \end{picture}%
\endgroup%
 
  \caption{The figure illustrates the possibility of the diffusion path going an arbitrary number of times around the torus, starting at the black dot and ending in the red. This is illustrated by the red path. The conditioning on single point in $\mathbb{T}^2$ therefore leads to conditioning on multiple points in $\mathbb{R}^2$. Left: Two paths from the same two-dimensional process with multiple endpoints. Right: The projection of the two paths onto the torus.}
   \label{fig:example}
\end{figure}

\section{Existence of Strong Solution}\label{sec: existence of solution}

The drift term in equation \eqref{bridge diffusion on torus} is discontinuous. However, we below show that it posses certain regularity conditions and use this to show that a strong solution to the SDE exist.

\fancybreak{}

In order to ensure the existence of a solution to the diffusion in \eqref{bridge diffusion on torus}, we need some regularity of the drift term. The drift coefficient is given by
\begin{align}\label{drift coefficient equation}
 1_{G^c}(X_t) \frac{\alpha(X_t)-X_t}{T-t} =
 \begin{cases}
 \frac{\alpha(X_t)-X_t}{T-t}, & \text{if } X_t \in G^c \\
 0, & \text{otherwise, }
\end{cases}
\end{align}
for every $ 0 \leq t < T$, where the superscript $c$ denotes the complement.
It is a discontinuous process with the set of discontinuities being the set $G$ consisting of the set of straight lines in $\mathbb{R}^2$ where the argmin process is non-unique. It is not even clear that the drift term is suitably measurable as the argmin map in general is not.

\begin{lemma}\label{measurable lemma}
Let $b \colon [0,T) \times \mathbb{R}^2 \rightarrow \mathbb{R}^2$ be the map given by \eqref{drift coefficient equation}. Then $b$ is $\mathcal{B}([0,T)) \otimes \mathcal{B}\big( \mathbb{R}^2\big) - \mathcal{B}\big( \mathbb{R}^2\big)$ measurable. Furthermore, the map $(s,\omega) \mapsto b(s,X_s(\omega))$ is $\mathcal{B}([0,t]) \otimes \mathcal{F}^0_t$ measurable, for every $0 \leq t < T$, where $(\mathcal{F}^0_t)$ denotes the natural filtration generated by $X$. This is called progressive measurability.
\end{lemma}

\begin{proof}
First note that $G^c$ is a Borel measurable set as we can write write it as a countable union of open sets, i.e., for $y=(y_1,y_2)$ we have
\begin{equation*}
 G^c = \bigcup_{y \in \pi^{-1}(a)} \bigl(y_1-\tfrac{1}{2},y_1+ 
 \tfrac{1}{2}\bigr) \times \bigl(y_2 -\tfrac{1}{2}, y_2 +  
 \tfrac{1}{2}\bigr) =: \bigcup_{y \in \pi^{-1}(a)} V_y.
\end{equation*}
Now, we need to show that for all $A \in \mathcal{B}\big(\mathbb{R}^2\big)$, the set $b^{-1}(A)$ is an element of $\mathcal{B}([0,T)) \otimes \mathcal{B}\big( \mathbb{R}^2\big)$. 
It is enough to consider all open subsets $U \subseteq \mathbb{R}^2$ as these sets generate the Borel algebra on $\mathbb{R}^2$. So let $U$ be an arbitrary open subset, then we have that 
\begin{equation*}
b^{-1}(U) = b^{-1}(U) \cap \big([0,T) \times G^c\big) \cup b^{-1}(U) \cap \big([0,T) \times G\big).
\end{equation*}
As $b$ is continuous on each of the sets $[0,T) \times V_y$ we have that $b^{-1}(U) \cap \big([0,T) \times G^c\big)$ is a countable union of open sets and therefore an element of $\mathcal{B}([0,T)) \otimes \mathcal{B}\big( \mathbb{R}^2\big)$. For the second part we see that
\begin{equation*}
b^{-1}(U)\cap ([0,T) \times G) =
\begin{cases}
[0,T) \times G, & \text{if } (0,0) \in U  \\
\emptyset, & \text{otherwise, }
\end{cases}
\end{equation*}
where both are elements of $\mathcal{B}([0,T)) \otimes \mathcal{B}\big( \mathbb{R}^2\big)$. This shows that $b$ is Borel measurable. 

\fancybreak{}

Progressive measurability follows by a very similar argument.
\end{proof}

\noindent Usually, global or local Lipschitz conditions are imposed on the drift and diffusion coefficients in order to secure global (resp. local) strong solutions to an SDE. This is a too strong condition for the drift term in this case, however, it is bounded in the following sense.

\begin{lemma}\label{bounded drift lemma}
The drift coefficient in \eqref{drift coefficient equation} is uniformly bounded in $x$ and in $t$ on $[0,S]$, for any $0 \leq S < T$.
\end{lemma}

\begin{proof}
The first assertion is clear. Let $S \in [0,T)$ be arbitrary and $ 0 \leq t \leq S$. For every $x \in G^c$ there exist a $y \in \pi^{-1}(a)$ such that we have
\begin{equation*}
 \norm*{ 1_{G^c}(x)\frac{\alpha(x)-x}{T-t}}^2 = \norm*{ 
 \frac{y-x}{T-t}}^2 
 \leq \frac{C}{(T-S)^2} = C_S,
\end{equation*}
for some positive constants $C > 0$.
\end{proof}

\noindent We now come to the main result of this section.
\begin{proposition}
There exist a strong solution of \eqref{bridge diffusion on torus} on $[0,T)$, which is strongly unique.
\end{proposition}

\begin{proof}
The drift term is Borel measurable and bounded on $[0,S]$ by \reflemma{measurable lemma} and \reflemma{bounded drift lemma}. As indicated in \cite[Thm. 2]{veretennikov1980strong} and \cite[Thm.1]{veretennikov1981strong} \eqref{bridge diffusion on torus} has a strong solution which is strongly unique. 
\end{proof}

\begin{remark}
The assumption in \cite[Thm. 2]{veretennikov1980strong} can be verified by using smooth bump functions.
\end{remark}

\section{Convergence and Absolute Continuity}\label{sec: absolute continuity}

The considerations above make the solution of \eqref{bridge diffusion on torus} into a continuous semimartingale. If a semimartingale $X$ takes its values in an open set $U$ of $\mathbb{R}^2$ then Itô's formula holds true for any $C^{1,2}([0,T) \times U)$ functions as well.

\begin{proposition}
Let $X$ be a solution to \eqref{bridge diffusion on torus} on the filtered probability space $(\Omega, \mathcal{F},(\mathcal{F}_t), P)$. For every $\omega \in \Omega$ for which there exist an $S<T$ such that $X_t(\omega)$ stays in $G^c$ on $[S,T)$, then $X$ converges pointwise almost surely to $\pi^{-1}(a)$.
\end{proposition}

\begin{proof}
Assume that for some $\omega \in \Omega$ there exist some $S<T$ such that on $[S,T)$ the process $X_t(\omega)$ takes its values in $G^c$. By continuity of the process it will take it its values in some open neighborhood $V_y$ of the point $y \in \pi^{-1}$. The proof is then identical to the proof in \cite[Lemma 4]{delyon_simulation_2006}.
\end{proof}

\begin{remark}
It is of course of interest to show that for almost every path the process will converge. This can be obtained by showing that the process will not intersect $G$ infinitely many times close to $T$.
\end{remark}
 
\noindent Consider the stochastic process $\mathcal{E}$ on $ 0 \leq t \leq S$ defined by
\begin{equation}\label{Dolean dade exponential}
 \mathcal{E}(L)_t= \exp\Bigl( -\int_0^t b(s,X_s) dW_s - 
 \frac{1}{2} \int_0^t \norm{b(s,X_s)}^2ds\Bigr),
\end{equation}
where $L$ is the local martingale in the exponential. This is known as the Doléans-Dade exponential.
From \reflemma{bounded drift lemma} it follows that, for all $t \leq S$,
\begin{equation*}
 \mathbb{E}\Bigl[\exp\Bigl(\int_0^t \norm{b(s,X_s)}^2 ds
 \Bigr)\Bigr]  \leq  \exp(t C_S) < \infty
\end{equation*}
The above is known as the Novikov condition (cf. \cite{novikov_identity_1973}) which ensures that \eqref{Dolean dade exponential} is a martingale on $[0,T)$. Girsanov's theorem (\cite[Thm. 5.1 Chap. 3]{karatzas1991brownian}) then provides that the process defined by
\begin{equation*}
\widetilde{W}_t = W_t + \int_0^t b(s,X_s)ds
\end{equation*}
is a Brownian motion under the new measure $Q$ introduced below.

\begin{theorem}
Let $X$ defined on $(\Omega, \mathcal{F}, (\mathcal{F}_t),P)$ be a solution of \eqref{bridge diffusion on torus} on $[0,S]$ for $S<T$. The process in \eqref{Dolean dade exponential} defined on $0 \leq t \leq S$ ($S<T$) is a true martingale and so there exists a measure $Q$ which is absolutely continuous wrt. $P$ such that $X$ is $Q$-Brownian motion.
\end{theorem}

\begin{proof}
The martingale property of \eqref{Dolean dade exponential} on $[0,S]$ is a consequence of the Novikov condition. Then Girsanov's theorem gives us that $X$ is a $Q$-Brownian motion on $[0,S]$.
\end{proof}

From the (perhaps obvious) fact that the distribution of the true Brownian bridge is locally equivalent to the distribution of the Brownian motion up to time $t < T$, it follows that the distribution of the Brownian bridge is absolutely continuous wrt. the proposed process up to time $t < T$.

\begin{remark}
A bit of extra work is needed to obtain the correction term as in \cite{delyon_simulation_2006}. There are indications that it is possible to simulate from the true distribution of the Brownian bridge on the torus, however, Theorem 1 shows that \eqref{bridge diffusion on torus} can approximate it. 
\end{remark}

\section{Numerical Experiments.}\label{sec: numerical experiments}

For the numerical implementation of the proposed SDE in equation \eqref{bridge diffusion on torus} we implemented the Euler-Maruyama scheme, i.e. taking $n$ equidistant discretization points of the time interval $t_1,...,t_n$, with $t_{i+1}-t_i=\Delta t$, the numerical equation becomes
\begin{align*}
x_{t_{i+1}} = x_{t_i} + \frac{\argmin_{y \in \pi^{-1}(a)}(\norm{y-x_{t_i}})-x_{t_i}}{T-t_i} \Delta t + \sigma \Delta W_{t_i},
\end{align*}
where $\Delta W_{t_{i+1}} = W_{t_{i+1}}-W_{t_i}$ is equal in distribution to a normal random variable with mean zero and variance $\Delta t$.

\begin{figure}[!bt]
\centering
\begin{subfigure}[!bt]{0.4\textwidth}
  \includegraphics[scale=0.1]{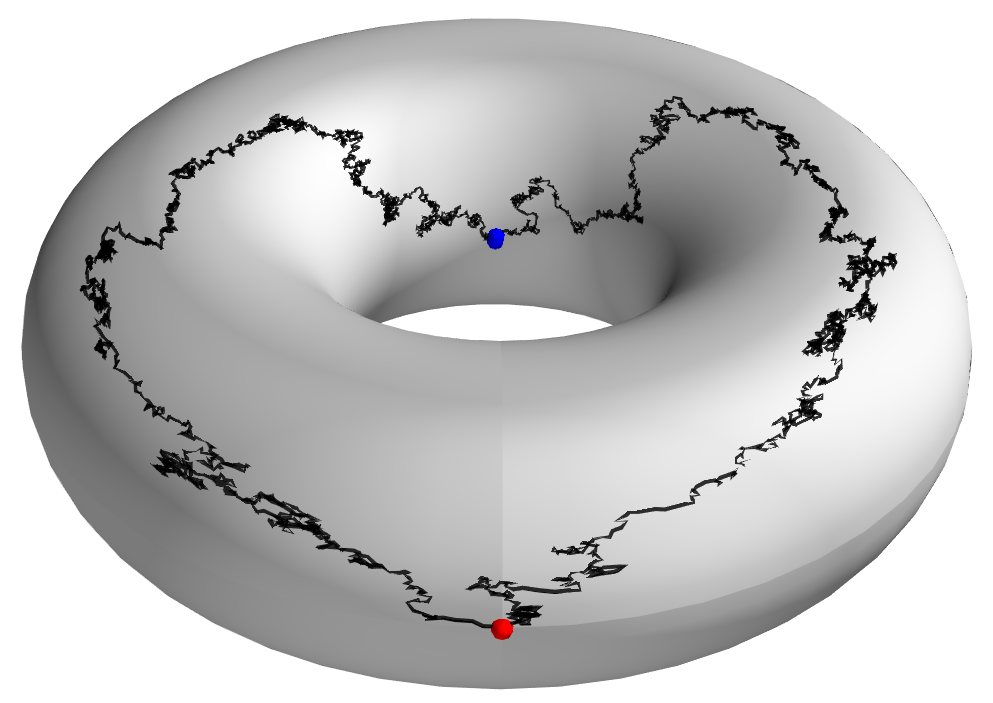} 
\caption{Paths visualized on an embedded torus.}
   \label{fig: bridge diffusion on torus}
\end{subfigure}
\qquad
\begin{subfigure}[!bt]{0.4\textwidth}
	\includegraphics[scale=0.13]{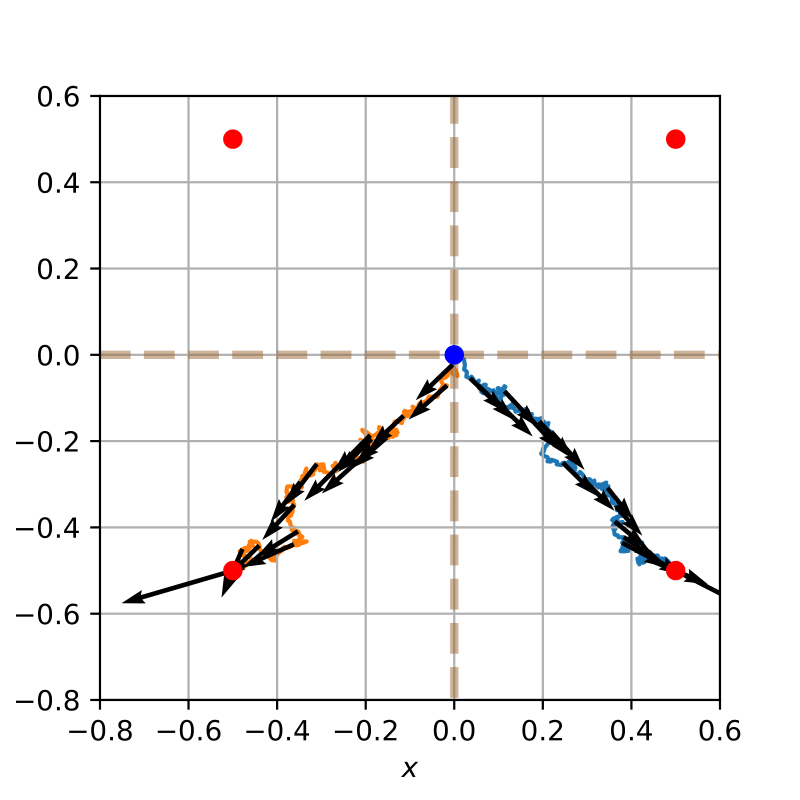}
	\caption{The two Euclidean paths that are mapped onto the torus.}
	\label{fig: two euc proc}
\end{subfigure}
\caption{Two different paths visualized both on the torus and in Eucliden space. The blue dot represents the starting point and the red represents the end point.}
\end{figure}

Figure \ref{fig: bridge diffusion on torus} shows the implementation of the numerical scheme on an embedded torus and Figure \ref{fig: two euc proc} its Euclidean counterpart. Figure \ref{fig: drift} shows the behaviour of the drift term along a given path, illustrating that the attraction becomes stronger as time approaches the terminal time. The vector fields in Figure \ref{fig: vector field} shows the constant attraction to the center of the open subsets.

\begin{figure}[!ht]
	\centering
	\begin{subfigure}[!bt]{0.4\textwidth}
		\includegraphics[scale=0.3]{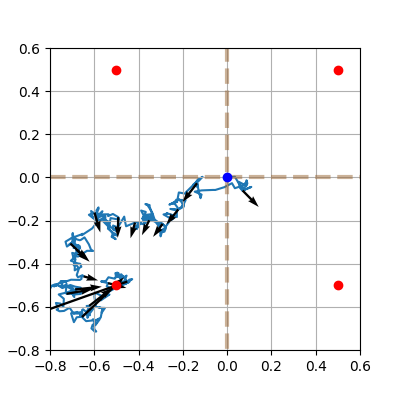}
		\caption{Drift term}
		\label{fig: drift}
	\end{subfigure}
	\qquad
	\begin{subfigure}[!bt]{0.4\textwidth}
		\includegraphics[scale=0.3]{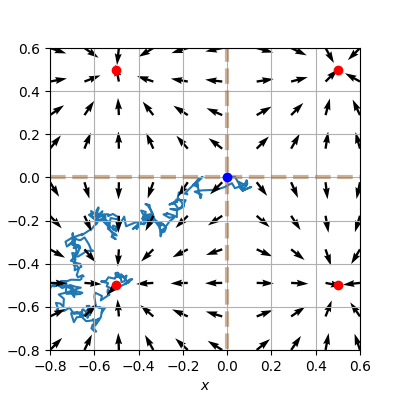}
		\caption{Vector field}
		\label{fig: vector field}
	\end{subfigure}
	\caption{Figure \ref{fig: drift} depicts the evolution of the drift term. It shows how the pull from the drift becomes stronger near the end. Figure \ref{fig: vector field} shows the underlying vector field.}
\end{figure}

\begin{figure}[!bt]
	\centering
		\includegraphics[scale=0.4]{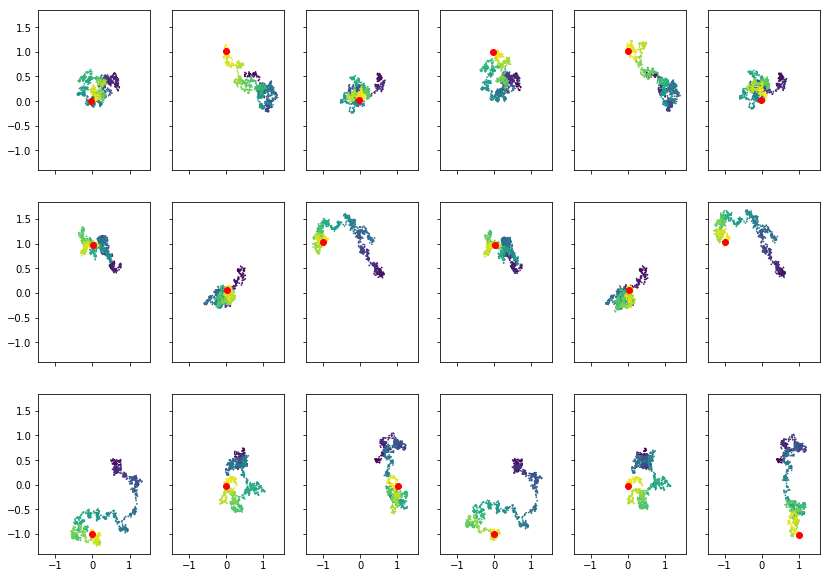}
	\caption{Figure \ref{fig: true vs proposed} shows 9 paths from the proposed model \eqref{bridge diffusion on torus} on the left and the corresponding paths from the true bridge \eqref{true bridge process} on the right. It is seen that the first and the last paths disagree on the limiting point, whereas the rest looks fairly similar. The picture agree with the fact that roughly four in five have the same limiting point. Here $\sigma = 0.8$ and the conditioning points being the integers in $[-2,2]\times [-2,2]$.}\label{fig: true vs proposed}
\end{figure}

\clearpage
\vspace{10pt}\hspace{-15pt}\textbf{Acknowledgements.} We acknowledge F. van der Meulen for discussions and insights on conditioned diffusions.

%
% ---- Bibliography ----
%
% BibTeX users should specify bibliography style 'splncs04'.
% References will then be sorted and formatted in the correct style.
%
\bibliographystyle{splncs04}
\bibliography{bibfile}

\end{document}